\documentclass[a4paper,11pt]{amsart}
\usepackage{amsfonts,amsthm,amsmath,amssymb,graphicx, overpic}
\theoremstyle{plain}
\newtheorem{lem}{Lemma}[section]

\newtheorem{thm}[lem]{Theorem}

\theoremstyle{definition}

\theoremstyle{remark}

\numberwithin{equation}{section}

\newcommand{\N}{\mathbb N}

\renewcommand{\epsilon}{\varepsilon}

\usepackage[usenames]{color}
\title[Subexponential sums of partial quotients]{Subexponentially increasing sums of partial quotients in
continued fraction expansions}
\date{}

\author{Lingmin Liao}
\address{Lingmin Liao\\LAMA UMR 8050, CNRS,
Universit\'e Paris-Est Cr\'eteil, 61 Avenue du
G\'en\'eral de Gaulle, 94010 Cr\'eteil Cedex, France}
\email{lingmin.liao@u-pec.fr}
\author{Micha\l\ Rams}
\address{Micha\l\ Rams\\Institute of Mathematics\\ Polish
Academy of Sciences\\ ul.
\'Sniadeckich 8, 00-656 Warszawa\\ Poland }
\email{rams@impan.pl}
\thanks{M.R. was partially supported by the MNiSW grant N201
607640 (Poland).}
\thanks{L.L. was partially supported by 12R03191A
- MUTADIS (France).}

\begin{document}
\begin{abstract}
We investigate from a multifractal analysis point of view the increasing rate of the sums of partial quotients $S_n(x)=\sum_{j=1}^n a_j(x)$, where $x=[a_1(x), a_2(x), \cdots ]$ is the continued fraction expansion of an irrational $x\in (0,1)$.  Precisely, for an increasing function $\varphi: \mathbb{N} \rightarrow \mathbb{N}$, one is interested in the Hausdorff dimension of the sets
\[
E_\varphi = \left\{x\in (0,1): \lim_{n\to\infty} \frac {S_n(x)} {\varphi(n)} =1\right\}.
\]
Several cases are solved by Iommi and Jordan, Wu and Xu, and Xu. We attack the remaining subexponential case $\exp(n^\gamma), \ \gamma \in [1/2, 1)$.  We show that when $\gamma \in [1/2, 1)$, $E_\varphi$ has Hausdorff dimension $1/2$. Thus, surprisingly, the dimension has a jump from $1$ to $1/2$ at $\varphi(n)=\exp(n^{1/2})$. In a similar way, the distribution of the largest partial quotient is also studied.
\end{abstract}

\maketitle
\def\thefootnote{}
\def\thefootnote{\arabic{footnote}}

\section{Introduction}

 Each irrational
number $x\in [0,1)$ admits a unique infinite continued fraction
expansion of the form
\begin{eqnarray}\label{ff1}
x=\frac{\displaystyle 1}{\displaystyle a_1(x)+ \frac{\displaystyle
1}{\displaystyle a_2(x)+\frac{\displaystyle 1}{\displaystyle
a_3(x)+\ddots}}},
\end{eqnarray}
where the positive integers $a_n(x)$ are called the partial quotients of $x$. Usually,  (\ref{ff1}) is written as $x=[a_1,a_2,\cdots]$ for simplicity.
 The $n$-th finite truncation of (\ref{ff1}): $p_n(x)/q_n(x)=[a_1,\cdots, a_n]$ is called the $n$-th convergent of $x$.
The continued fraction expansions can be induced by the Gauss
transformation $T:[0,1)\to [0,1)$ defined by  $$ T(0):=0, \ \text{and } \
T(x):=\frac{1}{x} \ {\rm{(mod \ 1)}}, \ {\rm{for}}\ x\in (0,1).
$$ It is well known that $a_1(x)= \lfloor x^{-1}\rfloor$ ($\lfloor \cdot \rfloor $ stands for the integer part)
  and $a_n(x)=a_1(T^{n-1}(x))$ for $n\geq 2$.

For any $n\geq 1$, we denote by $S_n(x)=\sum_{j=1}^n a_j(x)$ the sum of the $n$ first partial quotients. It was proved by Khintchine \cite{Kh35} in 1935 that $S_n(x)/(n\log n)$ converges in measure (Lebesgue measure) to the constant $1/\log 2$. In 1988, Philipp \cite{Ph88} showed that there is no reasonable normalizing sequence $\varphi(n)$ such that a strong law of large numbers is satisfied, i.e., $S_n(x)/\varphi(n)$ will never converge to a positive constant almost surely.

From the point of view of multifractal analysis, one considers the Hausdorff dimension of the sets
\[
E_\varphi = \left\{x\in (0,1): \lim_{n\to\infty} \frac {S_n(x)} {\varphi(n)} =1\right\}.
\]
where $\varphi :\N\rightarrow \N$ is an increasing function.

The case $\varphi(n)=\gamma n$ with $\gamma\in [1,\infty)$ was studied by
Iommi and Jordan \cite{IJ}.  It is proved that with respect to $\gamma$, the Hausdorff dimension (denoted by $\dim_H$) of $E_\varphi$ is analytic, increasing from $0$ to $1$, and tends to $1$ when  $\gamma$ goes to infinity. In \cite{WX11}, Wu and Xu proved that if $\varphi(n)=n^{\gamma}$ with $\gamma \in (1,\infty)$ or $\varphi(n)=\exp(n^\gamma)$ with $\gamma\in (0,1/2)$, then $\dim_HE_\varphi=1$. Later, it was shown by Xu \cite{Xu}, that if $\varphi(n)=\exp(n)$ then $\dim_HE_\varphi=1/2$ and if $\varphi(n)=\exp(\gamma^n)$ with $\gamma>1$ then $\dim_HE_\varphi=1/(\gamma+1)$. The same proofs of \cite{Xu} also imply that for $\varphi(n)=\exp(n^\gamma)$ with $\gamma\in (1, \infty)$ the Hausdorff dimension $\dim_HE_\varphi$ stays at $1/2$. So, only the subexponentially increasing case: $\varphi(n)=\exp(n^\gamma), \gamma\in [1/2, 1)$ was left unknown.
In this paper, we fill this gap.

\begin{thm}\label{main}
Let $\varphi(n)=\exp(n^\gamma)$ with $\gamma\in [1/2, 1)$. Then
\[
\dim_H E_\varphi=\frac 12.
\]
\end{thm}

We also show that there exists a jump of the Hausdorff dimension of $E_\varphi$ between $\varphi(n)=\exp(n^{1/2})$ and slightly slower growing functions, for example $\varphi(n)=\exp({\sqrt{n}(\log n)^{-1}})$.
\begin{thm}\label{main-2}
Let $\varphi(n)=\exp({\sqrt{n}\cdot \psi (n)})$ be an increasing function with $\psi$ being a $\mathcal{C}^1$ positive function on $\mathbb{R}_+$ satisfying
\begin{align}\label{ass-jump}
\lim_{x\to \infty} \frac{\sup_{y\geq x} \psi(y)^2}{\psi(x)} =0 \quad \text{and} \quad \lim_{x\to \infty}{x\psi'(x) \over \psi(x)}=0.
\end{align}
 Then
\[
\dim_H E_\varphi=1.
\]
\end{thm}

We remark that the assumption (\ref{ass-jump}) on the function $\psi$ says that $\psi$ decreases to $0$ slower than any polynomial.
We also remark that when $\psi$ is decreasing, then the first condition of (\ref{ass-jump}) is automatically satisfied.

Theorems \ref{main} and \ref{main-2} show that, surprisingly, there is a jump of the Hausdorff dimensions from $1$ to $1/2$ in the class $\varphi(n) = \exp(n^\gamma)$ at $\gamma = 1/2$ and that this jump cannot be easily removed by considering another class of functions. See Figure 1 for an illustration of the jump of the Hausdorff dimension.

\begin{figure}
\includegraphics[width=0.8\textwidth]{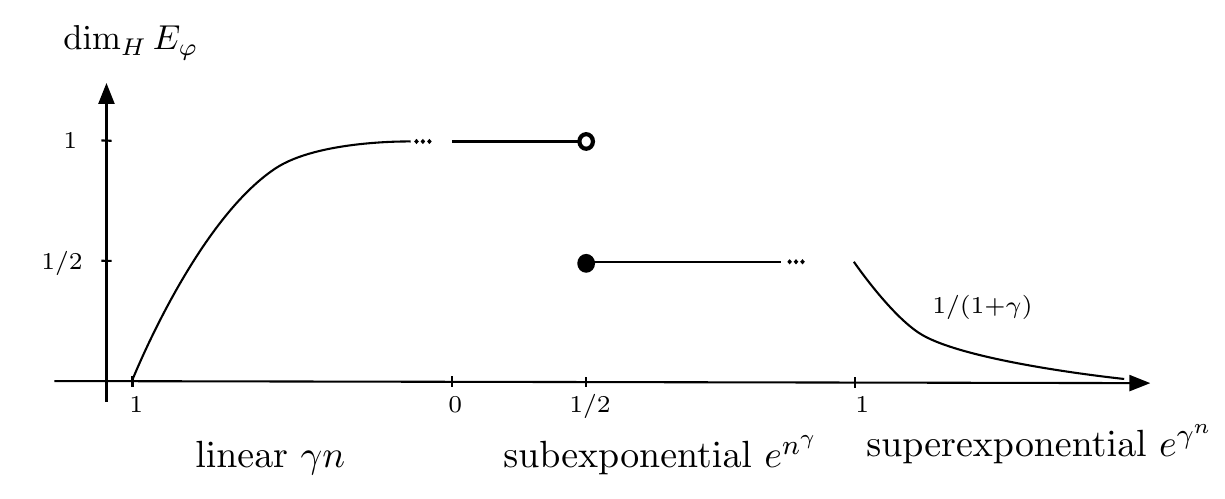}
\caption{$\dim_H E_{\varphi}$ for different $\varphi$.}
\end{figure}


\medskip
By the same method, we also prove some similar results on the distribution of the largest partial quotient in continued fraction expansions. For $x\in [0,1)\setminus \mathbb{Q}$, define
\[T_n(x) :=\max\{a_k(x): \ 1\leq k \leq n\}. \]
One is interested in the following lower limit: \[ T(x):=\liminf_{n\to\infty}\frac{T_n(x)\log\log n}{n}.  \]
It was conjectured by Erd\"os that almost surely $T(x)=1$.
However, it was proved by Philipp \cite{Ph75} that for almost all $x$, one has $T(x)=1/\log 2$. Recently, Wu and Xu \cite{WX09} showed that
\[
\forall \alpha\geq 0,\quad \dim_H\left\{x\in [0,1)\setminus \mathbb{Q} : \lim_{n\to\infty}\frac{T_n(x)\log\log n}{n}=\alpha \right\}=1.
\]
They also proved that if the denominator $n$ is replaced by any polynomial the same result holds. In this paper, we show the following theorem.
\begin{thm}\label{largest}
For all $\alpha >0$,
\[
F(\gamma, \alpha) = \left\{x\in [0,1)\setminus \mathbb{Q} : \lim_{n\to\infty}{T_n(x)}/{\exp({n^\gamma})}=\alpha \right\}
\]
satisfies
\[
\dim_H F(\gamma, \alpha) =
\begin{cases} 1 &\text{ if }\gamma \in (0,1/2)\\\frac 12&\text{ if }\gamma \in (1/2, \infty).\end{cases}
\]
\end{thm}
We do not know what happens in the case $\gamma=1/2$.

\section{Preliminaries}

For any $a_1, a_2,\cdots, a_n\in
\mathbb{N}$, call
\begin{equation*} I_n(a_1, \cdots, a_n):=\{x\in [0,1): a_1(x)=a_1,
\cdots, a_n(x)=a_n\}
\end{equation*}
 a {\em rank-$n$ basic interval}. Denote by $I_n(x)$ the rank-$n$ basic interval containing $x$. Write $|I|$ for the length of an
interval $I$.
The length of the basic interval $I_n(a_1, a_2, \cdots, a_n)$ satisfies
\begin{eqnarray}\label{length-estimate}
\prod\limits_{k=1}^n(a_k+1)^{-2}\leq \Big|I_n(a_1,\cdots, a_n)\Big| \leq
\prod\limits_{k=1}^na_k^{-2}.
\end{eqnarray}

Let $A(m,n):=\big\{(i_1, \dots, i_n)\in \{1,\dots, m\}^n: \ \sum_{k=1}^n i_k=m\big\}$. Let $\zeta(\cdot)$ be the Riemann zeta function.

\begin{lem}\label{estimation-zeta}
For any $s\in(1/2, 1)$, for all $n\geq 1$ and for all $m\geq n$, we have
\[
   \sum_{(i_1, \dots, i_n)\in A(m,n)} \prod\limits_{k=1}^ni_k^{-2s} \leq \left(\frac{9}{2}\big(2+\zeta(2s)\big)\right)^n m^{-2s}.
\]
\end{lem}
\begin{proof}
The proof goes by induction. First consider the case $n=2$. For $m=2$ the assertion holds, assume that $m>2$. We will estimate the sum $\sum_{i=1}^{m-1} i^{-2s}(m-i)^{-2s}$. For any $u\in [1,m/2]$ we have
\begin{eqnarray*}
   & &\sum_{i=1}^{m-1} i^{-2s}(m-i)^{-2s}
   =2\sum_{i=1}^{u-1} i^{-2s}(m-i)^{-2s} +\sum_{i=u}^{m-u} i^{-2s}(m-i)^{-2s}\\
  & \leq & 2\Big(\sum_{i=1}^{u-1} i^{-2s}\Big) (m-u)^{-2s} + (m-2u+1)u^{-2s} (m-u)^{-2s}\\
   &\leq & 2\zeta(2s)(m-u)^{-2s} + (m-2u+1)u^{-2s} (m-u)^{-2s}.
\end{eqnarray*}
Take $u=\lfloor m/3\rfloor$. Then one has
\[
(m-2u+1)u^{-2s} = (m+1)u^{-2s}-2u^{1-2s}\leq (m+1)\big\lfloor\frac{m}{3}\big\rfloor^{-2s}-2\leq 4.
\]
Hence, the above sum is bounded from above by
$$(4+2\zeta(2s))\cdot \big(\frac{2m}{3}\big)^{-2s} \leq \frac{9}{2}(2+\zeta(2s))\cdot m^{-2s}.$$

Suppose now that the assertion holds for $n\in \{2,n_0\}$. Then for $n=n_0+1$, we have
\begin{align*}
   & \sum_{(i_1, \dots, i_{n_0+1})\in \{1,\dots, m\}^{n_0+1}, \ \sum i_k=m} \prod\limits_{k=1}^{n_0+1}i_k^{-2s}  \\
   =&\sum_{i=1}^{m-1}i^{-2s} \sum_{(i_1, \dots, i_{n_0})\in \{1,\dots, m\}^{n_0}, \ \sum i_k=m-i} \prod\limits_{k=1}^{n_0}i_k^{-2s} \\
  \leq & \sum_{i=1}^{m-1} i^{-2s}\left(\frac{9}{2}\big(2+\zeta(2s)\big)\right)^{n_0}(m-i)^{-2s}\\
  =&\left(\frac{9}{2}\big(2+\zeta(2s)\big)\right)^{n_0} \cdot\sum_{i=1}^{m-1} i^{-2s}(m-i)^{-2s}\\
   \leq & \left(\frac{9}{2}\big(2+\zeta(2s)\big)\right)^{n_0}\cdot \left(\frac{9}{2}\big(2+\zeta(2s)\big)\right)m^{-2s}\\
   =&\left(\frac{9}{2}\big(2+\zeta(2s)\big)\right)^{n_0+1}m^{-2s}.
\end{align*}
\end{proof}

Let
\[
A(\gamma, c_1, c_2, N) := \left\{x\in (0,1) : c_1 < \frac {a_n(x)} {e^{n^\gamma}} < c_2, \ \forall n \geq N\right\}.
\]
Denote by $N_0$ the smallest integer $n$ such that $(c_2-c_1)\cdot e^{n^\gamma}>1$. Then the set $A(\gamma, c_1, c_2, N) $ is non-empty when $N\geq N_0$.

\begin{lem} \label{lem:biga}
For any $\gamma>0$, any $N\geq N_0$ and any $0<c_1<c_2$,
\[
\dim_H A(\gamma, c_1, c_2, N) = \frac 12.
\]
\end{lem}
\begin{proof}
This lemma is only a simple special case of \cite[Lemma 3.2]{FLWW}, but we will sketch the proof (based on \cite{JR}), needed for the next lemma. Without loss of generality, we suppose $N_0=1$ and let $N=1$ (the proof for other $N$ is almost identical).

Let $a_1, a_2,\ldots, a_n$ satisfy $c_1 < a_j e^{-j^\gamma} < c_2$ for all $j$. Those are exactly the possible sequences for which the basic interval $I_n(a_1,\ldots,a_n)$ has nonempty intersection with $A(\gamma, c_1, c_2, 1)$.

There are approximately
\begin{equation}\label{estimate-nb}
\prod_{j=1}^n (c_2-c_1)e^{j^\gamma} \approx e^{\sum_1^n j^\gamma}
\end{equation}
of such basic intervals, each of diameter
\begin{equation}\label{estimate-l}
|I_n(a_1,\ldots,a_n)| \approx e^{-2 \sum_1^n j^\gamma},
\end{equation}
(both estimations are up to a factor exponential in $n$). Hence, by using the intervals $\{I_n(a_1,\ldots,a_n)\}$ as a cover, we obtain
\[
\dim_H A(\gamma, c_1, c_2, 1) \leq \frac 12.
\]

To get the lower bound, we consider a probability measure $\mu$ uniformly distributed on $A(\gamma, c_1, c_2, 1)$, in the following sense: given $a_1, \ldots, a_{n-1}$, the probability of $a_n$ taking any particular value between $c_1 e^{n^\gamma}$ and $c_2 e^{n^\gamma}$ is the same.

The basic intervals $I_n(a_1,\ldots,a_n)$ have, up to a factor $c^n$, the length $\exp({-2\sum_1^n j^\gamma})$ and the measure $\exp({-\sum_1^n j^\gamma})$. They are distributed in clusters: all $I_n(a_1,\ldots,a_n)$ contained in a single $I_n(a_1,\ldots,a_{n-1})$ form an interval of length $\exp({n^\gamma})\cdot \exp({-2\sum_1^{n} j^\gamma})$ (up to a factor $c^n$, with $c$ being a constant), then there is a gap, then there is another cluster. Hence, for any $r\in (\exp({-2 \sum_1^n j^\gamma}), \  \exp({-2 \sum_1^{n-1} j^\gamma}))$ and any $x\in A(\gamma, c_1, c_2, 1)$ we can estimate the measure of $B(x,r)$:

\[
\mu(B(x,r)) \approx \begin{cases} r \cdot e^{- \sum_1^n j^\gamma} &\text{ if } r < e^{-2 \sum_1^{n} j^\gamma + n^\gamma}\\ e^{- \sum_1^{n-1} j^\gamma}  &\text{ if } r > e^{-2 \sum_1^{n} j^\gamma + n^\gamma}\end{cases}
\]
(up to a factor $c^n$). The minimum of $\log \mu(B(x,r))/\log r$ is thus achieved for $r=e^{-2 \sum_1^{n} j^\gamma + n^\gamma}$, and this minimum equals

\[
\frac {-\sum_1^{n-1} j^\gamma} {-2\sum_1^{n} j^\gamma + n^\gamma} \approx \frac {- n^{\gamma+1}/(\gamma+1)} {-2n^{\gamma+1}/(\gamma+1) - n^\gamma} = \frac 12 - O(1/n).
\]
Hence, the lower local dimension of $\mu$ equals $1/2$ at each point of $A(\gamma, c_1, c_2, 1)$, which implies
\[
\dim_H  A(\gamma, c_1, c_2, 1) \geq \frac 12
\]
by the Frostman Lemma (see \cite[Principle 4.2]{Fa}).
\end{proof}

Let now $c_1$ and $c_2$ not be constant but depend on $n$:
\[
B(\gamma, c_1, c_2, N) = \left\{x\in (0,1) : c_1(n) < \frac {a_n(x)} {e^{n^\gamma}} < c_2(n) \ \forall n \geq N\right\}.
\]
 A slight modification of the proof of Lemma \ref{lem:biga} gives the following.

\begin{lem} \label{lem:bigb}
Fix $\gamma>0$. Assume $0< c_1(n) <c_2(n)$ for all $n$. Assume also that
\[
\lim_{n\to \infty} \frac {\log (c_2(n) - c_1(n))} {n^\gamma}  =0
\]
and
\begin{equation*}
\liminf_{n\to \infty} {\log c_1(n) \over \log n} > - \infty \quad \text{and} \quad \limsup_{n\to \infty} {\log c_2(n) \over \log n} < + \infty.
\end{equation*}
Then there exists an integer $N_1$ such that $(c_2(n)-c_1(n))\cdot e^{n^\gamma}>1$ for all $n\geq N_1$, and for all $N\geq N_1$,
\[
\dim_H B(\gamma, c_1, c_2, N) = 1/2.
\]
\end{lem}
\begin{proof}
We need only to replace the constants $c_1$ and $c_2$ by $c_1(n)$ and $c_2(n)$ in the proof of Lemma \ref{lem:biga}. Notice that by the assumptions of Lemma \ref{lem:bigb}, the formula (\ref{estimate-nb}) holds up to a factor $\exp(\epsilon\sum_1^n j^\gamma)$ for a sufficiently small $\epsilon>0$. While the formula (\ref{estimate-l}) holds up to a factor $\exp({cn\log n})$ for some bounded $c$. All these factors are much smaller than the main term $\exp({\sum_1^n j^\gamma})$ which is of order $\exp({n^{1+\gamma}})$. The rest of the proof is the same as that of Lemma \ref{lem:biga}.
\end{proof}

\bigskip

\section{Proofs}

\begin{proof}[Proof of Theorem \ref{main}]
Let $\varphi: \mathbb{N} \rightarrow \mathbb{N}$ be defined by $\varphi(n)=\exp(n^\gamma)$ with $\gamma>0$. For this case, we will denote $E_\varphi$ by $E_\gamma$.

Let us start from some easy observations, giving (among other things) a simple proof of $\dim_H E_\gamma = 1/2$ for $\gamma\geq 1$.

Consider first $\gamma\geq 1/2$. If $x\in E_\gamma$
then for any $\epsilon >0$ and for $n$ large enough
\begin{align}\label{S_n-ineq} (1-\epsilon)e^{n^\gamma} \leq S_n(x) \leq (1+\epsilon)e^{n^\gamma}
\end{align}
and
\[
 (1-\epsilon)e^{(n+1)^\gamma} \leq S_{n+1}(x) \leq (1+\epsilon)e^{(n+1)^\gamma}.
\]
Hence
\[
 (1-\epsilon)e^{(n+1)^\gamma} -  (1+\epsilon)e^{n^\gamma} \leq a_{n+1}(x) \leq  (1+\epsilon)e^{(n+1)^\gamma} -  (1-\epsilon)e^{n^\gamma}.
\]
For $\gamma\geq 1$ this implies
\[
E_\gamma \subset \bigcup_N A(\gamma, c_1, c_2, N)
\]
for some constants $c_1, c_2$. By Lemma \ref{lem:biga},
\[
\dim_H E_\gamma \leq \frac 12, \quad \forall \gamma\geq 1.
\]

Consider now any $\gamma>0$. Set
$$c_1(n) = (e^{n^\gamma}-e^{(n-1)^\gamma})e^{-n^\gamma}\quad \text{and} \quad c_2(n)={n+1 \over n} c_1(n).$$
For $\gamma \geq 1$, $c_1(n)$ and $c_2(n)$ are bounded from below. For $\gamma <1$ and $n$ large, we have
$$
(e^{n^\gamma}-e^{(n-1)^\gamma})e^{-n^\gamma} \approx \gamma n^{\gamma-1}.
$$
Thus, in both cases the assumptions of Lemma \ref{lem:bigb} are satisfied. Checking  $B(\gamma, c_1, c_2, N) \subset  E_\gamma$, we deduce by Lemma \ref{lem:bigb} that
\[
\dim_H E_\gamma \geq \frac 12, \quad \forall \gamma>0.
\]
Therefore, we have obtained $\dim_H E_\gamma = 1/2$ for $\gamma\geq 1$ and $\dim_H E_\gamma \geq 1/2$ for $\gamma>0$.
What is left to prove is that for $\gamma\in [1/2, 1)$ we have $\dim_H E_\gamma \leq 1/2$.

Let us first assume that $\gamma>1/2$. Remember that if $x\in E_\gamma$, then for any $\epsilon>0$ and for $n$ large enough we have (\ref{S_n-ineq}).
Take a subsequence $n_0=1$, and $n_k=k^{1/\gamma}\ (k\geq 1$). Then there exists an integer $N\geq 1$ such that for all $k\geq N$,
\[
(1-\epsilon)e^{n_k^\gamma} \leq S_{n_k}(x) \leq (1+\epsilon)e^{n_k^\gamma},
\]
and (as $\exp(n_k^\gamma) = e^k$)
\[
(1-\epsilon)e^k-(1+\epsilon)e^{k-1} \leq S_{n_k}(x) -S_{n_{k-1}}(x) \leq (1+\epsilon)e^k-(1-\epsilon)e^{k-1}.
\]
Thus
\[
E_\gamma \subset \bigcup_N \bigcap_{k\geq N} A(\gamma, k, N),
\]
with $A(\gamma, k, N)$ being the union of the intervals $\{I_{n_k}(a_1, a_2, \cdots, a_{n_k})\}$ such that
\[
\sum_{j=n_{\ell-1}+1}^{n_\ell} a_j=m \quad \text{with}\quad m\in D_\ell, \quad N\leq \ell \leq k,
\]
where $D_\ell:=[(1-\epsilon)e^{n_\ell^\gamma}-(1+\epsilon)e^{n_\ell^\gamma-1}, (1+\epsilon)e^{n_\ell^\gamma}-(1-\epsilon)e^{n_\ell^\gamma-1}]$.

Now, we are going to estimate the upper bound of the Hausdorff dimension of $E_\varphi^{(1)} = \bigcap_{k} A(\gamma, k, 1)$. For $E_\varphi^{(N)} = \bigcap_{k\geq N} A(\gamma, k, N)$ with $N\geq 2$ we have the same bound and the proofs are almost the same.

Observe that every set $A(\gamma, k, N)$ has a product structure: the conditions on $a_i$ for $i\in (n_{\ell_1}, n_{\ell_1+1}]$ and for $i\in (n_{\ell_2}, n_{\ell_2+1}]$ are independent from each other. Hence,
for any $s\in (1/2, 1)$ we can apply  Lemma \ref{estimation-zeta} together with the formula
\[
 |I_{n_k}|^{s} \leq \prod_{\ell=1}^k (a_{n_{\ell-1}+1}a_{n_{\ell-1}+2}\cdots a_{n_{\ell}})^{-2s}
\]
to obtain
$$
\sum_{I_{n_k}\subset A(\gamma, k, 1)} |I_{n_k}|^{s} \leq \prod_{\ell=1}^k \sum_{m\in D_\ell} \left(\frac{9}{2}\big(2+\zeta(2s)\big)\right)^{n_\ell-n_{\ell-1}} m^{-2s}.
$$

Denote $r_1:=2 \epsilon (1-e^{-1})$ and $r_2:=(e-1-\epsilon e - \epsilon)/e$. Then we have
$|D_\ell| \leq  r_1e^{\ell} $ and any $m\in D_\ell$ is not smaller than $r_2 e^\ell$. Thus we get
\begin{align}\label{estimation-upp}
\sum_{I_{n_k}\subset A(\gamma, k, 1)} |I_{n_k}|^{s} \leq \prod_{\ell=1}^k r_1 e^\ell \cdot\left(\frac{9}{2}\big(2+\zeta(2s)\big)\right)^{\ell^{1/\gamma}-(\ell-1)^{1/\gamma}}\cdot r_2^{2s}e^{-2s\ell}.
\end{align}
We have $\ell^{1/\gamma}-(\ell-1)^{1/\gamma}\approx \ell^{1/\gamma-1}$. As $\gamma>1/2$, we have $1/\gamma-1<1$, and the main term in the above estimate is
$e^{(1-2s)\ell}$. Thus for any $s>1/2$, the product is uniformly bounded. Thus $\dim_HE_\varphi^{(1)} \leq 1/2$.

If $\gamma=1/2$, we take $n_k= k^2/L^2$ with $L$ being a constant and we repeat the same argument. Observe that now $\exp(n_k^\gamma) = e^{k/L}$. Then the same estimation will lead to
\begin{align}\label{estimation-upp1}
\sum_{I_{n_k}\subset A(\gamma, k, 1)} |I_{n_k}|^{s} \leq \prod_{\ell=1}^k r_1 r_2^{2s} \cdot  \left(\frac{9}{2}\big(2+\zeta(2s)\big)\right)^{\frac {\ell^{2}-(\ell-1)^{2}}{L^2}}e^{(1-2s)\ell/L}.
\end{align}
The main term of the right side of the above inequality should be $$\left(\frac{9}{2}\big(2+\zeta(2s)\big)\right)^{2\ell/L^2} \cdot e^{(1-2s)\ell/L}.$$
We solve the equation
\[ \left(\frac{9}{2}\big(2+\zeta(2s)\big)\right)^{2/L^2} \cdot e^{(1-2s)/L}=1,\]
which is equivalent to
\begin{equation}\label{eq} \left(\frac{9}{2}\big(2+\zeta(2s)\big)\right)=e^{\frac{2s-1} 2 L}.\end{equation}
Observe that the graphs of the two sides of \eqref{eq} (as functions of the variable $s$) always have a unique intersection for some $s_L\in [1/2,1]$, when $L$ is large enough. These $s_L$ are upper bounds for the Hausdorff dimension of $E_\varphi^{(1)}$. Notice that the intersecting point $s_L\to 1/2$ as $L\to \infty$ since the zeta function $\zeta$ has a pole at $1$. Thus the dimension of
 $E_\varphi^{(1)}$ is not greater than $1/2$.

 So, in both cases, we have obtained $\dim_H E_\gamma \leq 1/2$.
\end{proof}

\begin{proof}[Sketch proof of Theorem \ref{main-2}]
The proof goes like Section 4 of \cite{WX11} with the following changes. We choose $\epsilon_k = \psi(k)$. Let $n_1$ be such that $\varphi(n_1)\geq 1$ and define $n_k$ as the smallest positive integer such that
\begin{align}\label{varphi_k}
 \varphi(n_k)\geq (1+\epsilon_{k-1})\varphi(n_{k-1}).
 \end{align}
For a large enough integer $M$, set
\begin{align*}
E_M(\varphi):=&\Big\{ x\in [0,1): a_{n_1}(x) =\lfloor (1+\epsilon_1)\varphi(n_1)\rfloor+1, \\
&a_{n_k}(x)= \lfloor (1+\epsilon_k)\varphi(n_k)\rfloor -\lfloor (1+\epsilon_{k-1})\varphi(n_{k-1})\rfloor+1 \ \text{for all } k\geq 2,\\
& \text{and } 1\leq a_i(x)\leq M \ \text{for } i\neq n_k \ \text{for any } k\geq 1 \Big\}.
\end{align*}
We can check that $E_M(\varphi)\subset E_\varphi$.

To prove $\dim_HE_\varphi=1$, for any $\epsilon>0$, we construct a $(1/(1+\epsilon))$-Lipschitz map from $E_M(\varphi)$ to $E_M$, the set of numbers with partial quotients less than some $M$ in its continued fraction expansion. The theorem will be proved by letting $\epsilon\to 0$ and $M\to \infty$.

Such a Lipschitz map can be constructed by send a point $x$ in $E_M(\varphi)$ to a point $\tilde{x}$ by deleting all the partial quotients $a_{n_k}$ in its continued fraction expansion.
Define $r(n):=\min\{k: n_k \leq n\}$.
The $(1/(1+\epsilon))$-Lipschitz property will be assured if
\begin{align}\label{condition-digits-0}
\lim_{n\to\infty}\frac{r(n)}{n}=0,
\end{align}
and
\begin{align}\label{condition-digits}
\lim_{n\to\infty} \frac{\log(a_{n_1}a_{n_2}\cdots a_{n_{r(n)}})}{n}=0.
\end{align}

 In fact, by (\ref{ass-jump}), we can check for any $\delta>0$, $\psi(n)\leq n^{\delta}$ for $n$ large enough. Thus by definition of $n_k$, we can deduce that $r(n)\leq n^{1/2+\delta}$. Hence (\ref{condition-digits-0}) is satisfied.

 Further, we have
\begin{align}\label{sum-e-k}
\sum_{k=1}^{r(n)} \epsilon_k \approx r(n) \psi(r(n)).
\end{align}
By (\ref{varphi_k})
\[
\varphi(n)\geq \varphi(n_{r(n)}) \geq \prod_{k=1}^{r(n)-1}(1+\epsilon_k) \varphi(n_1) \geq e^{\sum_{k=1}^{r(n)} \epsilon_k/2}\varphi(n_1).
\]
Thus (\ref{sum-e-k}) implies
\begin{align}\label{rn-psi-rn}
r(n) \psi(r(n)) \ll \sqrt{n} \psi(n),
\end{align}
where $a_n\ll b_n$ means that $a_n/b_n$ is bounded by some constant when $n\to\infty$.

On the other hand, by (\ref{length-estimate}) and (\ref{varphi_k}), we have
\begin{align*}
\log(a_{n_1}a_{n_2}\cdots a_{n_{r(n)}}) \leq r(n) \log (2\varphi(n))+\sum_{k=1}^{r(n)} \epsilon_k.
\end{align*}
Hence (\ref{sum-e-k}) and (\ref{rn-psi-rn}) give
\[
\log(a_{n_1}a_{n_2}\cdots a_{n_{r(n)}}) \ll r(n)\sqrt{n}\psi(n) + r(n)\psi(r(n)) \ll {n \psi^2(n) \over \psi(r(n))} + r(n). 
\]
Finally, (\ref{condition-digits}) follows from the assumption (\ref{ass-jump}) and the already proved formula (\ref{condition-digits-0}).
\end{proof}


\begin{proof}[Proof of Theorem \ref{largest}]
For the case $\gamma<1/2$, the set constructed in Section 4 of \cite{WX11} (as a subset of the set of points for which $S_n(x) \approx e^{n^\gamma}$) satisfies also $T_n(x) \approx e^{n^\gamma}$ and has Hausdorff dimension one. We proceed to the case $\gamma> 1/2$.

The lower bound is a corollary of Lemma \ref{lem:bigb}. Take $c_1(n)= \alpha(1-{1 \over n})$ and $c_2(n)=\alpha$. Let $N_1$ be the smallest integer $n$ such that ${\alpha \over n}e^{n^\gamma}>1$. Then the conditions of Lemma \ref{lem:bigb} are satisfied, and for all points $x$ such that
$c_1(n)e^{n^\gamma}< a_n(x)< c_2(n)e^{n^\gamma}$, we have
\[T_n(x)/e^{n^\gamma} \geq c_1(n)=\alpha\left(1-{1 \over n}\right),\]
and
\[
T_n(x)/e^{n^\gamma} = a_{k} /e^{n^\gamma} \leq \alpha e^{k^\gamma}/e^{n^\gamma}  \leq \alpha,
\]
where $k\leq n$ is the position at which the sequence $a_1,\ldots, a_n$ achieves a maximum.
Thus for all $x\in B(\gamma, c_1, c_2, N_1)$
\[ \lim_{n\to \infty} T_n(x)/e^{n^\gamma} = \alpha. \]
Hence, $B(\gamma, c_1, c_2, N_1)\subset F(\gamma, \alpha) $ and the lower bound follows directly from Lemma \ref{lem:bigb}.

The upper bound is a modification of that of Theorem \ref{main}. 
 We consider the case $\alpha=1$ only, since for other $\alpha>0$, the proofs are similar.

Notice that for any $\epsilon>0$, if $x\in F(\gamma, 1)$, then for $n$ large enough,
\[ (1-\epsilon)e^{n^\gamma} \leq S_n(x) \leq n(1+\epsilon)e^{n^\gamma}.\]
Take a subsequence $n_k=k^{1/\gamma}(\log k)^{1/\gamma^2}$. Then
\[
(1-\epsilon)e^{k (\log k)^{1/\gamma}} \leq S_{n_k}(x) \leq k^{1/\gamma}(\log k)^{1/\gamma^2} (1+\epsilon)e^{k (\log k)^{1/\gamma}},
\]
and
\[
u_k\leq  S_{n_k}(x) -S_{n_{k-1}}(x) \leq v_k,
\]
with
\begin{align*}
u_k:=(1-\epsilon)e^{k (\log k)^{1/\gamma}} -(k-1)^{1/\gamma}(\log (k-1))^{1/\gamma^2} (1+\epsilon)e^{(k-1) (\log (k-1))^{1/\gamma}},
\end{align*}
and
\[
v_k:=k^{1/\gamma}(\log k)^{1/\gamma^2} (1+\epsilon)e^{k (\log k)^{1/\gamma}}-(1-\epsilon)e^{(k-1) (\log (k-1))^{1/\gamma}}.
\]
We remark that
\begin{align}\label{ukvk}
u_k > \frac 12 e^{k (\log k)^{1/\gamma}} , \ \ \ v_k < \frac 32 k^{1/\gamma}(\log k)^{1/\gamma^2} e^{k (\log k)^{1/\gamma}}
\end{align}
 when $k$ is large enough.

Observe that
\[
F(\gamma, 1) \subset \bigcup_N B(\gamma, N),
\]
with $B(\gamma, N)$ being the union of the intervals $\{I_{n_k}(a_1, a_2, \cdots, a_{n_k})\}_{k\geq N}$ such that
\[
\sum_{j=n_{\ell-1}+1}^{n_\ell} a_j=m \quad \text{with}\quad m\in D_\ell, \quad N \leq \ell \leq k,
\]
where $D_\ell$ is the set of integers in the interval $[u_\ell, v_\ell]$.

As in the proof of Theorem  \ref{main}, we need only study the set $B(\gamma,1)$.
For any $s\in (1/2, 1)$, since
\[
 |I_{n_k}|^{s} \leq \prod_{\ell=1}^k (a_{n_{\ell-1}+1}a_{n_{\ell-1}+2}\cdots a_{n_{\ell}})^{-2s},
\]
by Lemma \ref{estimation-zeta},
\begin{align*}
&\sum_{I_{n_k}\subset B(\gamma, N)} |I_{n_k}|^{s} \leq \prod_{\ell=1}^k \sum_{m\in D_\ell} \left(\frac{9}{2}\big(2+\zeta(2s)\big)\right)^{n_\ell-n_{\ell-1}} m^{-2s}.
\end{align*}
Note that by (\ref{ukvk}) the number of integers in $D_\ell$ satisfies
\[|D_{\ell}| \leq v_\ell-u_\ell \leq v_\ell < \frac 32 \cdot \ell^{1/\gamma}(\log \ell)^{1/\gamma^2}.\]
By (\ref{ukvk}), we also have \[m\geq u_\ell >\frac 12 e^{\ell (\log \ell)^{1/\gamma}} \quad \text{for any } m\in D_\ell.\]
Similar to (\ref{estimation-upp}) and (\ref{estimation-upp1}), we deduce that $\sum_{I_{n_k}\subset B(\gamma, N)} |I_{n_k}|^{s}$ is less than
\begin{align*}
 \prod_{\ell=1}^k \frac 32 \cdot \ell^{1/\gamma}(\log \ell)^{1/\gamma^2} e^{\ell (\log \ell)^{1/\gamma}}  \left(\frac{9}{2}\big(2+\zeta(2s)\big)\right)^{n_\ell - n_{\ell-1}} 2^{2s} e^{-2s\ell (\log \ell)^{1/\gamma}}.
\end{align*}

Since $n_\ell - n_{\ell-1}\approx \ell^{1/\gamma-1+o(\varepsilon)}$ and $1/\gamma-1<1$, the main term in the above estimation is
$e^{(1-2s)\ell (\log\ell)^{1/\gamma}}$. Thus for any $s>1/2$ the product is uniformly bounded and we have the Hausdorff dimension of $B(\gamma,1)$ is not greater than $1/2$. Then we can conclude $\dim_HF(\gamma,1) \leq 1/2$ and the proof is completed.
\end{proof}

\section{Generalizations}

In this section we consider after \cite{JR} certain infinite
iterated function systems that are natural generalizations of the Gauss map. For each $n \in \N$, let $f_n : [0, 1] \to [0, 1]$
be $C^1$ maps such that
\begin{itemize}
\item[(1)] there exists $m\in \N$ and $0 < A < 1$ such that for all $(a_1, . . . , a_m) \in \N^m$ and for all $x \in [0, 1]$
\[
0 < |(f_{a_1} \circ \cdots \circ f_{a_m})'(x)| \leq A < 1,
\]
\item[(2)] for any $i, j \in \N$ \  $f_i((0, 1)) \cap f_j((0, 1)) = \emptyset$,
\item[(3)] there exists $d > 1$ such that for any $\varepsilon > 0$ there exist $C_1(\varepsilon),C_2(\varepsilon)>0$ such that for $i \in \N$ there exist constants $\xi_i, \lambda_i$ such that for all $x \in [0, 1]$ \  $\xi_i \leq |f_i'(x)| \leq \lambda_i$ and
\[
\frac {C_1} {i^{d+\varepsilon}} \leq \xi_i \leq \lambda_i \leq \frac {C_2} {i^{d-\varepsilon}}.
\]
\end{itemize}
We will call such an iterated function system a {\it $d$-decaying system}. It will be further called {\it Gauss like} if
\[
\bigcup_{i=1}^\infty f_i([0,1]) = [0,1)
\]
and if for all $x\in [0,1]$ we have that $f_i(x) < f_j(x)$ implies $i<j$.

We have a natural projection $\Pi: \N^\N \to [0,1]$ defined by
\[
\Pi(\underline{a}) = \lim_{n\to \infty} f_{a_1}\circ \cdots \circ f_{a_n}(1),
\]
which gives for any point $x\in [0,1]$ its symbolic expansion $(a_1(x), a_2(x), \ldots)$. This expansion is not uniquely defined, but there are only countably many points with more than one symbolic expansions.

For a $d$-decaying Gauss like system we consider $S_n(x)=\sum_1^n a_i(x)$. Given an increasing function $\varphi: \N\to \N$ we denote
\[
E_d(\varphi) = \left\{x\in (0,1): \lim_{n\to\infty} \frac {S_n(x)} {\varphi(n)} =1 \right\}.
\]

\begin{thm}
Let $\{f_i\}$ be a $d$-decaying Gauss like system. We have
\begin{itemize}
\item[i)] if $\varphi(n) = e^{n^\gamma}$ with $\gamma < 1/d$,
\[
\dim_H E_d(\varphi) =1,
\]
\item[ii)] if $\varphi(n) = e^{n^\gamma}$ with $\gamma > 1/d$,
\[
\dim_H E_d(\varphi) =\frac 1d,
\]
\item[iii)] if $\varphi(n) = e^{\gamma^n}$ with $\gamma >1$,
\[
\dim_H E_d(\varphi) = \frac 1 {\gamma + d-1}.
\]
\end{itemize}
\end{thm}

The proofs (both from Section 3 and from \cite{WX11, Xu}) go through without significant changes.

\bibliographystyle{alpha}

%
%

\end{document}